\documentclass[12pt]{article}
\usepackage{graphicx}
\usepackage{amsmath, amsthm}
\usepackage{amsfonts}
\usepackage{amssymb}
\usepackage{url}
\usepackage{hyperref}  
\usepackage{verbatim}
\usepackage{dsfont} 
\usepackage{color}
\definecolor{red}{rgb}{1,0,0}

\definecolor{red}{rgb}{1,0,0}

\definecolor{cyan}{rgb}{0.05,.4,.7}

\usepackage[mathlines]{lineno}
\usepackage{enumerate}

\def\ord#1{| #1 |}

\newcommand{\R}{\mathbb{R}}

\newcommand{\N}{\mathbb{N}}

\newcommand{\tw}{\operatorname{tw}}
\newcommand{\h}{\eta}

\newcommand{\G}{\mathcal{G}}
\newcommand{\sRn}{S_n(\R)}
\newcommand{\sym}{\mathcal{S}}
\newcommand{\SG}{\sym(G)}

\newcommand{\bone}{\ensuremath{\mathds{1}}} 
\newcommand{\bxi}{{\vec \xi}}

\newcommand{\Zp}{\operatorname{Z}_+}
\newcommand{\Zl}{\operatorname{Z}_\ell}
\newcommand{\ppw}{\operatorname{ppw}}
\newcommand{\pw}{\operatorname{pw}}

\newcommand{\la}{\operatorname{la}}

\newcommand{\trt}{\operatorname{trt}}

\newcommand{\Zmm}{\lfloor \operatorname{Z}\rfloor}

\newcommand{\CCRZlmm}{\operatorname{CCR-}\Zlmm}
\newcommand{\CCRZpmm}{\operatorname{CCR-}\Zpmm}

\newcommand{\Zpmm}{\lfloor \Zp\rfloor}
\newcommand{\Zlmm}{\lfloor \Zl\rfloor}

\newcommand{\Gc}{\overline{G}}

\newcommand{\lc}{\left\lceil}
\newcommand{\rc}{\right\rceil}
\newcommand{\lf}{\left\lfloor}
\newcommand{\rf}{\right\rfloor}

\def\ngsu#1{#1}
\def\ngsl#1{\underline{#1}}
\def\ngpu#1{#1^\times}
\def\ngpl#1{\underline{#1}^\times}
\def\ngsun#1{#1_{\rm nd}}
\def\ngsln#1{\underline{#1}_{\rm nd}}
\def\ngpun#1{#1^\times_{\rm nd}}
\def\ngpln#1{\underline{#1}^\times_{\rm nd}}

\newcommand{\bit}{\begin{itemize}}
\newcommand{\eit}{\end{itemize}}
\newcommand{\ben}{\begin{enumerate}}
\newcommand{\een}{\end{enumerate}}
\newcommand{\beq}{\begin{equation}}
\newcommand{\eeq}{\end{equation}}
\newcommand{\bea}{\begin{eqnarray*}}
\newcommand{\eea}{\end{eqnarray*}}
\newcommand{\bpf}{\begin{proof}}
\newcommand{\epf}{\end{proof}\ms}
\newcommand{\ms}{\medskip}

\newtheorem{thm}{Theorem}[section]

\newtheorem{cor}[thm]{Corollary}
\newtheorem{lem}[thm]{Lemma}
\newtheorem{conj}[thm]{Conjecture}

\newtheorem{definition}[thm]{Definition}
\newtheorem{observation}[thm]{Observation}
\newtheorem{remark}[thm]{Remark}
\newtheorem{example}[thm]{Example}
\newtheorem{question}[thm]{Question}

\newenvironment{obs}{\begin{observation}\bgroup\rm }{\egroup\end{observation}}
\newenvironment{rem}{\begin{remark}\bgroup\rm }{\egroup\end{remark}}

\title{Multi-part Nordhaus-Gaddum type problems for tree-width, Colin de Verdi\`ere type parameters, and Hadwiger number}

\author{ 
Leslie~Hogben\thanks{Department of Mathematics, Iowa State
University, Ames, IA 50011, USA (hogben@iastate.edu, chlin@iastate.edu, myoung@iastate.edu).} \thanks{American Institute of Mathematics, 600 E. Brokaw Road, 
San Jose, CA 95112 (hogben@aimath.org).}
\and Jephian C.-H.~Lin\footnotemark[1]\and Michael Young\footnotemark[1]}

\begin{document}

\maketitle


\vspace{-10pt}
\begin{abstract} 
A traditional Nordhaus-Gaddum problem for a graph parameter $\beta$ is to find a (tight) upper or lower bound on the sum or product of $\beta(G)$ and $\beta(\Gc)$ (where $\Gc$ denotes the complement of $G$). An $r$-decomposition $G_1,\dots,G_r$ of the complete graph $K_n$  is a partition of the
edges of $K_n$ among $r$ spanning subgraphs $G_1,\dots,G_r$. A traditional Nordhaus-Gaddum problem can be viewed as the special case for $r=2$ of a more general $r$-part sum or product Nordhaus-Gaddum type problem.  We determine the values of the $r$-part sum and product  upper bounds asymptotically as $n$ goes to infinity for the parameters tree-width and its variants largeur d'arborescence, path-width, and proper path-width.  We also establish ranges for the lower bounds for these parameters, and ranges for the upper and lower bounds of the $r$-part Nordhaus-Gaddum type problems for the parameters Hadwiger number, the Colin de Verdi\`ere number $\mu$ that is used to characterize planarity, and its variants $\nu$ and $\xi$.

\end{abstract}

\noindent\textbf{Keywords.}   Nordhaus-Gaddum, multi-part, Hadwiger number, tree-width, largeur d'arborescence, path-width,  proper path-width, Colin de Verdi\`ere type parameter.

\noindent\textbf{AMS subject classifications.} 05C35, 05C83, 05C69, 05C50  

\section{Introduction}\label{sintro}

An {\em $r$-decomposition} of $K_n=([n],E)$ is a partition of the edges as the edge sets of $r$ spanning subgraphs $G_i=([n],E_i)$  
for $i=1,\dots,r$.
An $r$-part Nordhaus-Gaddum  problem for a graph parameter $\beta$ is to find a (tight) upper or lower bound on the sum
 \[  \beta(G_1)+\dots+\beta(G_r) \]
 or on the product
  \[  \beta(G_1)\cdots\beta(G_r). \]
It is common to exclude from these bounds graphs of  order $\le n_0$ for a fixed $n_0$,  
or to ask that the bound be approached for arbitrarily large $n$. 

The study of Nordhaus-Gaddum type problems for more than two parts was initiated by F\"uredi et al.\! in \cite{multiNG}, and we follow the definitions 
of that paper. Let $\beta$ be a nonnegative integer valued graph parameter, and let $r$ and $n$ be positive integers.  Define the Nordhaus-Gaddum sum upper bound and sum lower bound:
 \[ \ngsu{\beta}(r;n):=\max \left\{\beta(G_1)+\dots+\beta(G_r) \right\}\]  where the maximum is taken over all $r$-decompositions of $K_n$;
  \[ \ngsl{\beta}(r;n):=\min \left\{\beta(G_1)+\dots+\beta(G_r)\right\} \]  where the minimum is taken over all $r$-decompositions of $K_n$ (in \cite{multiNG} this lower value $\ngsl{\beta}$ is denoted by $\overline \beta$, but we believe a lower line is more mnemonic for a lower value).  

Define the Nordhaus-Gaddum product  upper bound and product lower bound:   
\[\ngpu{\beta}(r;n)=\max \left\{\beta(G_1) \cdots \beta(G_r) \right\}\]  where the maximum is taken over all  
$r$-decompositions of $K_n$;
\[\ngpl{\beta}(r;n)=\min \left\{\beta(G_1) \cdots  \beta(G_r) \right\}\]  where the minimum is taken over all  
$r$-decompositions of $K_n$.  For the product lower bound,  because many of the parameters take on the value zero on an edgeless  graph, we  focus on the non-degenerate bound $\ngpln{\beta}(r;n)$, in which every graph in a decomposition must have an edge (see Section \ref{ssNGdefs}).

Since $\ngsu{\beta}(1;n)=\beta(K_n)$ and similarly for the other bounds defined, we study decompositions with $r\ge 2$.  There is a rich literature on $2$-part Nordhaus-Gaddum type problems  (see \cite{NGsurvey} for a survey). We study $r$-part Nordhaus-Gaddum type problems for the following parameters: the Hadwiger number $\h$, tree-width $\tw$ and its variants largeur d'arborescence  $\la$, path-width $\pw,$ and proper path-width $\ppw,$    and  the  Colin de Verdi\`ere type parameters $\mu,\nu, \xi$, all of which are minor monotone and have interesting relationships (see \cite{param} for a discussion of these relations).  Definitions of the parameters are given in  Section \ref{ssparamdefs}.  
Kostochka says \cite[p. 307]{Kost84}, ``it is very important to study the Hadwiger number" due to its relation to other ideas in graph theory, including Hadwiger's famous conjecture.   
  He establishes extensive Nordhaus-Gaddum  theory ($r=2$) for $\eta$ in \cite{Kost81, Kost84, K89}. Likewise, tree-width and its variants have played a fundamental role in the theory of graph minors since their (re)introduction by Robertson and  Seymour in the early 1980s.   Surprisingly, Nordhaus-Gaddum  theory ($r=2$) has only recently been studied for tree-width and its variants, with the sum lower bound established in \cite{EGR11, JW12} and the sum upper bound established in \cite{JW12}. In addition to other uses in graph theory that motivated their introduction, Colin de Verdi\`ere type parameters have played an important role in the study of minimum rank/maximum nullity of real symmetric matrices described by a graph (see \cite{param, HLA2, H15IMA}).  The Nordhaus-Gaddum sum lower bound ($r=2$) is called the Graph Complement Conjecture in minimum rank literature (see \cite{HLA2}).

Here we state our main results. 
  We begin with the sum upper bound.

\begin{thm}\label{mainsumupper} $\null$\vspace{-4mm}
\ben[(a)]
\item\label{mainsumupper1} For a fixed $r\ge 2$ and $\beta$ one of  tree-width $\tw$ or its variants largeur d'arborescence  $\la$, path-width $\pw,$ or proper path-width $\ppw,$ \vspace{-2mm}
\[ \lim_{n\to\infty} \frac {\ngsu{\beta}(r;n)}n= r.\vspace{-2mm}\] 
\item\label{mainsumupper2} For a fixed $r\ge 2$ and  $\beta$ one of  the Hadwiger number $\h$ or the  Colin de Verdi\`ere type parameters $\mu,\nu, \xi$, $\ngsu{\beta}(r;n)=\Theta(n)$.  Specifically, for $n\ge 3$ as $n\to \infty$\vspace{-2mm}
\[\frac r {\lc \sqrt{2  r+\frac{1}{4}}-\frac{1}{2}\rc} -o(1)  \le  \frac {\ngsu{\beta}(r;n)}n\le \sqrt r+o(1).\vspace{-3mm}\] 
\een
\end{thm}

Theorem \ref{mainsumupper} is established by Theorems \ref{thmtwupper},  \ref{thmrxiupper},  \ref{thmretalower}, and Corollary \ref{corrnumulower}.  It should be noted that  Theorem \ref{mainsumupper}\eqref{mainsumupper1} was established for  $r=2$ and tree-width in \cite{JW12} and better results are known for Theorem \ref{mainsumupper}\eqref{mainsumupper2} in the case that $r=2$:  Kostochka \cite{Kost81} showed that for $n\ge 5$, $\ngsu{\h}(2;n)=\lf \frac 6 5 n\rf$, so $\lim_{n\to\infty}\frac {\ngsu{\h}(2;n)}n=\frac 6 5>1=\frac 2 {\lc \sqrt{2\cdot 2+\frac{1}{4}}-\frac{1}{2}\rc}$.  Barrett et al.\! \cite{upNG} showed that $\frac 4 3\le \limsup_{n\to\infty}\frac {\ngsu{\beta}(2;n)}n\le \sqrt 2$ for $\beta=\xi$ or $\nu$.  The fact that $\lim_{n\to\infty}\frac {\ngsu{\eta}(2;n)}n<\limsup_{n\to\infty}\frac {\ngsu{\beta}(2;n)}n$ for $\beta\in\{\xi,\nu\}$ suggests that the upper bound in Theorem \ref{mainsumupper}\eqref{mainsumupper2} is not tight for $\eta$ and the lower bound  is not tight for $\beta\in\{\xi,\nu\}$, since our upper bound for $\eta$ is obtained from our upper bound for  $\xi$ and our lower bound for $\nu$ and $\xi$ is obtained from our lower bound for  $\eta$. 

Next we consider the sum lower bound. 

\begin{thm}\label{mainetabarthm} For a fixed $r\ge 2$ and $n\to\infty$, $\ngsl{\h}(r;n)=\Theta( \frac n{\sqrt {\log n}})$.  Specifically, for $n$ large enough,\vspace{-2mm}
\[\frac 1 {570 r}\frac n{\sqrt {\log n}}\le \ngsl{\h}(r;n)\le  r \frac n{\sqrt {\log n}}.\vspace{-2mm}\]
  \end{thm}

Theorem \ref{mainetabarthm} is established by Theorem \ref{etabarthm} using results of Kostochka \cite{Kost84}.

 It is known \cite{EGR11, H15IMA, JW12} that $\ngsl{\beta}(2;n)=n-2$ for $\beta\in\{\tw,\la,\pw,\ppw\}$, and conjectured that $\ngsl{\beta}(2;n)=n-2$ for $\beta\in\{\mu,\nu,\xi\}$ (see \cite{ HLA2, H15IMA} and the references therein). 
  This does not generalize to $r>2$.

\begin{thm}\label{mainallsumlower}
For $\beta\in\{ \tw, \la, \pw, \ppw\}$ and $r\ge 3$, $\ngsl{\beta}(r;n)=\Theta(n)$.  Specifically, 
as $n\to \infty$,\vspace{-2mm}
\[\frac{1}{2}- o(1)<  r-\sqrt{r^2-r}-o(1)\leq \frac{\ngsl{\beta}(r;n)}n\le \frac 3 4 +o(1).\vspace{-2mm}\]
\end{thm}

Theorem \ref{mainallsumlower} is established by Corollaries \ref{thmallsumlower}  and  \ref{cor:twblower}.
Since $\eta(G)-1\le \nu(G)\le\xi(G)\le\ppw(G)$ and $\eta(G)-1\le \mu(G)\le\xi(G)\le\ppw(G)$ for any graph that has an edge, Theorems \ref{mainetabarthm} and \ref{mainallsumlower} imply lower and upper bounds on for $\beta\in\{ \mu,\nu,\xi\}$. 

Next we turn our attention to product bounds, beginning with the upper bound and followed by the non-degenerate lower bound.

\begin{thm}\label{mainprodupper} $\null$\vspace{-3.5mm}
\ben[(a)]
\item\label{mainprodupper1} For a fixed $r\ge 2$ and $\beta$ one of  tree-width $\tw$ or its variants 
$\la$, 
$\pw,$ or 
$\ppw,$ \vspace{-2mm}
\[ \lim_{n\to\infty} \frac {\ngpu{\beta}(r;n)}{n^r}= 1.\vspace{-2mm}\] 
\item\label{mainprodupper2} For a fixed $r\ge 2$ and  $\beta$ one of  the Hadwiger number $\eta$ or the  Colin de Verdi\`ere type parameters $\nu,\mu,\xi$, $\ngpu{\beta}(r;n)=\Theta(n^r)$. Specifically, as       $n\to\infty$,   \vspace{-2mm}  
\[\lc \sqrt{2  r+\frac{1}{4}}-\frac{1}{2}\rc^{-r}-o(1)\leq  \frac{\ngpu{\beta}(r;n)}{n^r}\le (\sqrt r)^{-r}+o(1).\vspace{-2mm}\]
\een
\end{thm}

Theorem \ref{mainprodupper} is established by Theorems \ref{thm:twpu} and  \ref{thm:CdVpu}. Again,  better results are known for Theorem \ref{mainsumupper}\eqref{mainsumupper2} in the case that $r=2$:  Kostochka \cite{Kost81} showed that for $n\ge 5$, $\ngsu{\h}(2;n)=\lf \frac 1 4\lf\frac{6 n}5\rf^2\rf$, so $\lim_{n\to\infty}\frac {\ngpu{\h}(2;n)}{n^2}=\frac 9{25}> \frac 1 4 =\lc \sqrt{2 \cdot 2+\frac{1}{4}}-\frac{1}{2}\rc^{-2}$.  Barrett et al.\! \cite{upNG} showed that $\frac 4 9\le \limsup_{n\to\infty}\frac {\ngsu{\beta}(2;n)}n\le \frac 1 2$ for $\beta=\xi$ or $\nu$.  Again the upper bound in Theorem \ref{mainsumupper}\eqref{mainsumupper2} is likely not tight for $\eta$ and the lower bound  is likely not tight for $\beta\in\{\xi,\nu\}$. 

\begin{thm}\label{mainprodlower} For a fixed $r\ge 2$ and $\beta$  one of the Hadwiger number $\eta$, the  Colin de Verdi\`ere type parameters $\nu,\mu,\xi$, tree-width $\tw$ or its variants   $\la,\pw$, or $\ppw$, 
$ {\ngpln{\beta}(r;n)}=\Theta(n).$ 
Specifically:\vspace{-3mm}
\ben[(a)] 
\item Let $\beta\in \{\tw,\la,\pw,\ppw\}$.  For $n\ge 4$, $\ngpln{\beta}(2;n)=n-3$ {\rm\cite{H15IMA}}.  For $r\ge 3$ as $n\to\infty$, \vspace{-2mm}
\[\frac{1}{2}-o(1)\leq \frac{\ngpln{\beta}(r;n)}n\leq 1,\vspace{-2mm}\]
\item For  $\beta\in\{\nu,\xi,\mu\}$, $r\ge 2$, and  $n\ge 1$,\vspace{-2mm}
 \[\frac 1 {2^{2r-2}}\leq\frac{\ngpln{\beta}(r;n)}n\le 1.\vspace{-2mm}\]
\item For $r\ge 2$ and   $n\ge 1$,\vspace{-2mm}
\[(0.513)^{r-2} \leq \frac{\ngpln{\h}(r;n)}n\leq 2^{r-1}.\vspace{-2mm}\]

\een

  \end{thm}

Theorem \ref{mainprodlower} is established by Theorem \ref{thm:twprod}  and Corollaries \ref{thm:etaprodnd} and \ref{cor:etaprodnd2munuxi}.
In the case of the Hadwiger number, the general lower bound, which permits graphs with no edges, is also of interest.

\begin{thm}\label{prodeta} For a fixed  $r\ge 2$, $ {\ngpl{\h}(r;n)}=\Theta(n).$  Specifically,\vspace{-2.5mm}
\[(0.513)^{r-2} \leq \frac{\ngpl{\h}(r;n)}n\leq 1.\vspace{-2mm}\]
   \end{thm}

Theorem \ref{prodeta} is established by Theorem \ref{thm:etaprod}, with the case $\ngpl{\h}(2;n)=n$ following from results of Kostochka \cite{K89} (see Remark \ref{rem:degenHadprodlower}).
We also make a conjecture about the product lower bound for $\eta$, which  is implied by Hadwiger's Conjecture (see Remark \ref{rem:etapl}).

\begin{conj}
\label{etapl}
For all  $r\ge 2$ and $n\ge 1$,  $\ngpl{\h}(r;n)=n$.
\end{conj}

\subsection{Definitions of parameters and notation}\label{ssparamdefs}

All graphs are simple, undirected, and finite,  $G$ denotes a graph, and $n$ denotes the  order of $G$. A more complete description of  the parameters discussed here (and the graph notation we use) can be found in  \cite{param} and \cite{H15IMA}.  
The \emph{clique number} $\omega(G)$ is the maximum order of a clique in $G$ and 
  the \emph{Hadwiger number} $\h(G)$ is the maximum order of  a clique minor of $G$.

A graph consisting of isolated vertices has tree-width zero (and similarly for the variants of tree-width defined here), so in the remainder of this paragraph we assume a graph has an edge.  Let $k$ be a positive integer.  A \emph{$k$-tree} is constructed inductively by starting with a
complete graph on $k+1$ vertices and connecting each new vertex to the vertices of an existing
clique on $k$ vertices, so a tree of order at least two is a 1-tree.
   The {\em tree-width} $\tw(G)$ 
  is the minimum $k$ for which $G$ is a subgraph of a $k$-tree.
Every $k$-tree has at least two vertices of degree $k$.    The maximal cliques of a $k$-tree are of order ${k+1}$, and the \emph{facets} of a  maximal clique are its $k$-clique subgraphs.  
A {\em linear $k$-tree}  is constructed
inductively by starting with $K_{k+1}$ and connecting each new vertex to
a facet that includes 
the vertex added in the previous step.      
The {\em proper path-width} $\ppw(G)$  
is the minimum $k$  for which $G$ is a subgraph of a linear $k$-tree.
A {\em $k$-caterpillar} is constructed  by starting with $K_{k+1}$ and at each stage adding a new maximal clique by adjoining a new vertex to the $k$ vertices of some facet of the maximal clique that was added in the previous step.   
     The {\em path-width} $\pw(G)$  
      is the minimum $k$ for which $G$ is a subgraph of a $k$-caterpillar.  
A {\em two-sided  $k$-tree}  is
constructed  by starting with $K_{k+1}$ and connecting each
new vertex to the vertices of an existing $K_{k}$ that either includes
a vertex of degree $k$ or is the same as the $K_{k}$ to which some
previous vertex was connected.
The {\em  largeur d'arborescence} $\la(G)$  
is the minimum $k$ for which $G$ is a subgraph of a two-sided $k$-tree. Clearly $\tw(G)\le\la(G)\le\pw(G)\le\ppw(G)$ for every graph $G$.  It is known that $\la(G)\le\tw(G)+1$ \cite{CdVnu} and $\ppw(G)\le\pw(G)+1$ \cite{TUK94} for every graph $G$.
  A more comprehensive discussion of these parameters, including justification for some of  the equivalent definitions used here, is given in \cite{param}.


The Colin de Verdi\`ere type parameters are linear algebraic graph parameters.  All matrices discussed are real and symmetric; the set of $n\times n$ real symmetric matrices is denoted by $\sRn$.   
For $A=[a_{ij}]\in \sRn$,
the \emph{graph} of $A$ is $\G(A)=([n],E)$ where 
  $E=\{\{i,j\} : i,j\in[n],  i\ne j, \mbox{ and }a_{ij} \ne 0  \}$; the diagonal of $A$ is ignored in determining $\G(A)$.
The \emph{set of symmetric matrices described by  $G$} is
$\SG=\{A\in\sRn : \G(A)=G\}.$  A real symmetric matrix $A$ satisfies the \emph{Strong Arnold
Property} provided there does {not} exist a nonzero real symmetric matrix
$X$ satisfying $AX = O$, $A\circ X = O$, and $I\circ X=O$, where $\circ$ denotes the entry-wise  product, i.e., $(A\circ B)_{ij}=a_{ij}b_{ij}$,  $I$ is the identity matrix, and $O$ is the zero matrix.
The parameter $\xi(G)$ is  the maximum nullity
among  matrices $A\in\sym(G)$  satisfying the Strong Arnold Property.  
The parameter $\nu(G)$ is  the maximum nullity
among positive semidefinite matrices $A\in\sym(G)$  satisfying the Strong Arnold Property (a matrix $A\in\sRn$ is \emph{positive semidefinite} if  
all eigenvalues of $A$ are nonnegative).  
The Colin de Verdi\`ere number $\mu(G)$ is defined to be the maximum nullity
among symmetric matrices $A=[a_{ij}]\in\sym(G)$ such that $A$ satisfies the Strong Arnold Hypothesis, $A$ has exactly one negative eigenvalue, and 
 $A$ is a generalized Laplacian (i.e., for all $i\ne j$, $a_{ij}\le 0$).  

Asymptotic comparisons arise naturally in our work, and since some of the notation has more than one interpretation in the papers cited, we state our notation here. Let $f$ and $g$ be 
real valued functions of $\N$. 
We say:
  $f$ is $o(g)$  if
$\lim_{n\to\infty}\left|\frac{f(n)}{g(n)}\right|=0$;
  $f$ is $O(g)$  if  there exist
constants $C,N$ such that
$|f(n)|\le C|g(n)|$ for all $n\ge N$;
 $f$ is $\Omega(g)$  if  there exist
constants $C,N$ such that
$|f(n)|\ge C|g(n)|$ for all $n\ge N$;
  $f$ is $\Theta(g)$  if $f$ is $O(g)$ and $\Omega(g)$. \vspace{-3mm}

\subsection{
Decompositions and non-degeneracy}\label{ssNGdefs}

A {\em random $r$-decomposition} is an $r$-decomposition of $K_n$ in which each $G_i$ is a $G(n,\frac 1 r)$ random graph (with the understanding that necessarily the $G_i$ are dependent).

\begin{lem}\label{randomrlem}   
For fixed $r$ and $n$ large, a random $r$-decomposition exists.
\end{lem}
\bpf Each edge of $K_n$ is assigned by an independent  $r$-valued variable that determines which of the graphs $G_i, i=1,\dots,r$, is assigned the edge.  Thus $G_i, i=1,\dots,r$, is an $r$-decomposition of $K_n$ and each $G_i$ has probability $\frac 1 r$ for each edge, i.e., $G_i$ is $G(n,\frac 1 r)$. \vspace{-5pt} \epf

Some of the results we use  require a graph to have an edge.  A {\em non-degenerate $r$-decomposition} of $K_n=([n],E)$ is a partition of the edges as the edge sets of $r$ spanning subgraphs $G_i=([n],E_i), i=1,\dots,r$ with  $E_i\ne \emptyset$ for $i=1,\dots,r$.  We define non-degenerate versions of the four bounds, $\ngsun{\beta}(r;n), \ngsln{\beta}(r;n), \ngpun{\beta}(r;n), $ and $\ngpln{\beta}(r;n)$, where the maximum or minimum is taken over all non-degenerate $r$-decompositions of $K_n$.  

Suppose $G_1,G_2,\ldots, G_r$ is an $r$-decomposition on $n$ vertices.  If in this decomposition there are exactly $\ell$ graphs each having at least one edge, we may assume that $G_{\ell+1}=\cdots =G_r=\overline{K_n}$ are empty graphs, so that $G_1,\ldots, G_{\ell}$ form a non-degenerate $\ell$-decomposition.    

\begin{obs}\label{rem:degenerate}
Assuming  $\beta(\overline{K_{n}})=0$ or $\beta(\overline{K_{n}})=1$, the relationship between degenerate and non-degenerate sum bounds is:  \vspace{-2mm} 
\[\ngsu\beta(r;n)=\left\{\begin{array}{cc}
\displaystyle \max_{1\leq \ell\leq n}\{\ngsun{\beta}(\ell;n)+(r-\ell)\beta(\overline{K_{n}})\} & \text{if }\beta(\overline{K_n})=1,\\\vspace{-1mm}
\max_{1\leq \ell\leq n}\{\ngsun{\beta}(\ell;n)\} & \text{if }\beta(\overline{K_n})=0;\\\vspace{-2mm}
\end{array}\right.\vspace{-2mm}\] 
\[\ngsl\beta(r;n)=\left\{\begin{array}{cc}
\displaystyle \min_{1\leq \ell\leq n}\{\ngsln{\beta}(\ell;n)+(r-\ell)\beta(\overline{K_{n}})\} & \text{if }\beta(\overline{K_n})=1,\\\vspace{-1mm}
\min_{1\leq \ell\leq n}\{\ngsln{\beta}(\ell;n)\} & \text{if }\beta(\overline{K_n})=0;\\\vspace{-1mm}
\end{array}\right.\vspace{-2mm}\] 
\end{obs}

\begin{obs}\label{rem:prodnondegen}
Assuming  $\beta(\overline{K_{n}})=0$ or $\beta(\overline{K_{n}})=1$, the relationship between degenerate and non-degenerate product bounds is:  \vspace{-2mm} 
\[\ngpu\beta(r;n)=\left\{\begin{array}{cc}
\displaystyle \max_{1\leq \ell \leq r}\{\ngpun{\beta}(\ell;n)\} & \text{if }\beta(\overline{K_n})=1,\\\vspace{-1mm}
\ngpun{\beta}(r;n) & \text{if }\beta(\overline{K_n})=0;\\\vspace{-2mm}
\end{array}\right.\vspace{-2mm}\] 
\[\ngpl\beta(r;n)=\left\{\begin{array}{cc}
\displaystyle \min_{1\leq \ell \leq r}\{\ngpln{\beta}(\ell;n)\} & \text{if }\beta(\overline{K_n})=1,\\\vspace{-1mm}
0 & \text{if }\beta(\overline{K_n})=0.\\
\end{array}\right. \vspace{-3mm}\]
\end{obs}
 
 In the case of the product lower bound for parameters having $\beta(\overline{K_{n}})=0$, the non-degenerate parameter is clearly the more interesting one.   

\section{Nordhaus-Gaddum Sum Upper Bounds}\label{sNGsumupper}

For all the parameters  $\beta\in\{\h,\mu,\nu,\xi, \tw,\la,\pw,\ppw \}$, $\beta(G)\le n$, so $\ngsu{\beta}(r;n)\le rn$.  For tree-width and its variants largeur d'arborescence, path-width and proper path-width, 
this is essentially the best we can do.  

\begin{thm}\label{thmtwupper}  For a fixed $r\ge 2$, $\ngsu{\tw}(r;n)=rn-o(n)=\ngsu{\la}(r;n)= \ngsu{\pw}(r;n)= \ngsu{\ppw}(r;n)$. \end{thm}
\bpf  For tree-width, this follows from the fact that for fixed $p>0$, $\tw(G(n,p))=n-o(n)$  \cite{PS},  the   existence of a random $r$-decomposition (Lemma \ref{randomrlem}), and the linearity of expectation (see, for example,  \cite[Lemma 8.5.7]{West}). The remaining statements follow from the fact that $\ppw(G)\ge\pw(G)\ge\la(G)\ge \tw(G)$ for all $G$. \epf

Next we consider the Colin de Verdi\`ere type parameters and Hadwiger number.    
\begin{thm}\label{thmrxiupper}  Let $\beta\in\{\xi,\nu,\mu\}$.  For a fixed $r\ge 2$ and $n\ge 2\sqrt r$,
$ \ngsu{\beta}(r;n)\le\sqrt r n$.  For a fixed $r\ge 2$ and $n\ge 2\sqrt r$,    $\ngsu{\h}(r;n)\le \sqrt r\, n+r$.
\end{thm}
\bpf 
We begin by establishing the statement for $\xi$.    
It is shown in \cite{HHMS} that for any graph $G=(V(G),E(G))$, $\ord{E(G)} \ge \frac {\xi(G)(\xi(G)+1)} 2 -1$.  In  \cite{upNG} it is noted that  this implies $\xi(G)^2  \le  2\ord {E(G)}$ for every graph $G$ that has an edge.  Let $G_1,\dots,G_r$ be a non-degenerate $r$-decomposition of $K_n$.  Let $\bone$ denote the all ones vector of length $r$ and $\bxi:=[\xi(G_1),\dots,\xi(G_r)]^T$.  Then 
\[\|\bxi\|_2^2=\xi(G_1)^2+\dots+\xi(G_r)^2\le 2(|E(G_1)|+\dots+|E(G_r)|)\le 2\left( \frac {n^2} 2\right)=n^2.\]
That is, $\|\bxi\|_2\le n$.  By the Cauchy-Schwarz inequality,
\[\xi(G_1)+\dots+\xi(G_r)=\bone^T\bxi\le \|\bone\|_2\|\bxi\|_2\le \sqrt r\, n,\]
so $ \ngsun{\xi}(r;n)\le\sqrt r\, n$.

Now assume $n\ge 2\sqrt r$.  By Observation \ref{rem:degenerate}, 
\[\ngsu{\xi}(r;n)=\max_{1\leq \ell\leq n}\{\ngsun{\xi}(\ell;n)+(r-\ell)\xi(\overline{K_{n}})\}\leq \max_{1\leq \ell\leq n}\{\sqrt \ell \,n+r-\ell\}\le \sqrt r\,n.\]
The last inequality follows from the assumption that $n\ge 2\sqrt r$ by elementary algebra.

Since $\nu(G)\le \xi(G)$ and $\mu(G)\le \xi(G)$ for all graphs $G$, $\ngsu{\nu}(r;n)\le\ngsu{\xi}(r;n)$ and $\ngsu{\mu}(r;n)\le\ngsu{\xi}(r;n)$ for all $r\ge 2$ and $n$.  
Since  $\h(G)-1\le\nu(G) $ for all graphs $G$, $\ngsu{\h}(r;n)\le\ngsu{\nu}(r;n)+r$ for all $r\ge 2$ and $n$. 
\epf

\begin{rem}\label{rem:n.vs.m}  For any minor monotone parameter $\beta$, 
$n\ge m$ implies $\ngsu{\beta}(r;n)\ge \ngsu{\beta}(r;m)$, because we can take a decomposition  $G_1,\dots,G_r$ that realizes $\ngsu{\beta}(r;m)$,  add the extra vertices to every $G_i$, and allocate the extra edges among the $G_i$ however we choose without lowering $\beta(G_1)+\dots+\beta(G_r)$.
\end{rem}

A {\em triangular number} is a number of the form $\frac{t(t+1)} 2$ for some positive integer $t$. 
For an arbitrary positive integer $r$, define the {\em triangular root}  of $r$ by $\trt(r):= \sqrt{2  r+\frac{1}{4}}-\frac{1}{2}$,  so  $t=\lc\trt(r)\rc$ is equivalent to $\frac{t(t+1)} 2< r\le \frac{(t+1)(t+2)} 2$. 

\begin{thm}\label{thmretalower}  $\null$ \vspace{-3mm}
\ben[(i)]
\item\label{thmretalower12} Fix $r\ge 2$  and define   $t:=\lc\trt(r)\rc$.  Then for  every $s\ge 1$ and  $n:=ts$,\vspace{-2mm}  \[\frac{r}{t}n+(r-t)\leq \ngsu{\h}(r;n).\vspace{-2mm}\]
\item\label{thmretalowernew} For a fixed $r\ge 2$  and $n\to\infty$, \vspace{-2mm}
\[\frac r {\lc \sqrt{2  r+\frac{1}{4}}-\frac{1}{2}\rc}\, n-o(n)\le  \ngsu{\h}(r;n).\vspace{-2mm}\]
\een\end{thm}
\bpf 

\eqref{thmretalower12}: Let $n=ts$.  Partition the vertices into  sets $V_i,i=1,\dots,t$ of $s$ vertices each.  For $i=1,\dots,t$ define $H_i$ to be the subgraph of $K_n$ induced by $V_i$ (isomorphic to $K_s$) together with the $n-s$ vertices not in $V_i$ (as isolated vertices).     For $i=1,\dots,t-1, j=i+1,\dots,t$ define $H_{i,j}$ to be the subgraph of $K_n$ consisting of all $n$ vertices and the edges having one end-point in $V_i$ and the other in $V_j$; $H_{i,j}$ is isomorphic to $K_{s,s}$ and isolated vertices.     This defines $\frac{t(t+1)}{2}$ subgraphs.  By the choice of $t$, $r\le\frac{t(t+1)} 2$.  This allows us to assign $G_i=H_i$ for $i=1,\ldots ,t$ and then pick $r-t$ distinct $H_{j,k}$'s as $\{G_i\}_{i=t+1}^r$.  Put every edge not in any of $\{G_i\}_{i=1}^t$ in $G_1$ so that $\{G_i\}_{i=1}^t$ forms an $r$-decomposition.  Now $\h(G_i)\ge \h(K_s)=s$ for $i=1,\ldots ,t$ and $\h(G_i)\geq \h(K_{s,s})=s+1$ for $s=t+1,\ldots ,r$ (with the latter done by contracting a matching of cardinality $s-1$ ).  Therefore,\vspace{-2mm}
\[\ngsu{\h}(r;n)\ge \sum_{i=1}^r\h(G_i)\ge ts+(r-t)(s+1)=rs+(r-t) \ge
\frac{r}{t} n+(r-t).\vspace{-2mm}\]

\eqref{thmretalowernew}: 
Define  $t:=\lc \trt(r)\rc$, and for $n\ge  t$, define  $q:=\lf \frac n {t}\rf$ and $m:=qt$.  Note that $n-m<t$.   By Remark \ref{rem:n.vs.m} and part  \eqref{thmretalower12},
\[\ngsu{\h}(r;n)\ge \ngsu{\h}(r;m)\ge \frac r t m = \frac r t n - \frac r t (n-m)\ge \frac r t n -  r=\frac r t n - o(n).\qedhere\]

\epf

Since $\h(G)-1\le\mu(G)\le \xi(G)$, $\ngsu{\h}(r;n)-r\le\ngsu{\mu}(r;n)\le\ngsu{\xi}(r;n)$ (and similarly for $\nu$),  the next   corollary is immediate.
\begin{cor}\label{corrnumulower} 
Let $\beta\in\{\mu,\nu,\xi\}$.   \vspace{-2mm}
\ben[(i)]
\item Fix $r\ge 2$  and define   $t:=\lc\trt(r)\rc$.  Then for  every $s\ge 1$ and  $n:=ts$,\vspace{-2mm}  \[\frac{r}{t}n-t\leq \ngsu{\beta}(r;n).\vspace{-2mm}\]
\item\label{corrnumulower3} For a fixed $r\ge 2$  and $n\to\infty$, \vspace{-2mm}
\[\frac r {\lc \sqrt{2  r+\frac{1}{4}}-\frac{1}{2}\rc}\, n-o(n)\le  \ngsu{\beta}(r;n).\vspace{-2mm}\]
\een \end{cor}

  Thus, $\frac r {\lc \sqrt{2  r+\frac{1}{4}}-\frac{1}{2}\rc}   \le  \limsup_{n\to\infty}\frac {\ngsu{\beta}(r;n)}n\le  \sqrt{r}$ for $\beta\in\{\xi,\nu,\mu,\eta\}$.  In Table \ref{tabsumupper} we provide the values of these lower and upper bounds for small $r$.  
  \vspace{-2mm}

\begin{table}[h!]\caption{Lower and upper bounds for $\displaystyle \limsup_{n\to\infty}\frac {\ngsu{\beta}(r;n)}n$ when $\beta\in\{\xi,\nu,\mu,\eta\}$}\label{tabsumupper}
\begin{center} 
\noindent \begin{tabular}{|r|r|r|}
\hline 
 $r$ &  $\frac r {\lc \sqrt{2  r+\frac{1}{4}}-\frac{1}{2}\rc}$ & $\sqrt r$  \\[3.3mm]
 \hline 
3 & 1.5 & 1.73205\\ \hline 
4 & 1.33333 & 2.\\ \hline
5 & 1.66667 & 2.23607\\ \hline
6 & 2. & 2.44949\\ \hline
7 & 1.75 & 2.64575\\ \hline
8 & 2. & 2.82843\\ \hline
9 & 2.25 & 3.\\ \hline
10 & 2.5 & 3.16228\\
\hline 
\end{tabular}
\end{center}
\end{table}

\section{Nordhaus-Gaddum Sum Lower Bounds}\label{sNGsumlower}

We use the ideas in Kostochka's proof that $\ngsl{\h}(2;n)=\Theta( \frac n{\sqrt {\log n}})$ \cite[Corollary 5]{Kost84} to establish an analogous result for $\ngsl{\h}(r;n)$ with $r\ge 3$.

\begin{thm}\label{etabarthm} For a fixed $r\ge 2$ and $n\to\infty$, $\ngsl{\h}(r;n)=\Theta( \frac n{\sqrt {\log n}})$.  Specifically, for $n$ large enough,
\[\frac 1 {570 r}\frac n{\sqrt {\log n}}\le \ngsl{\h}(r;n)\le  r \frac n{\sqrt {\log n}}.\]
  \end{thm}
\bpf  Let  $G_i=(V_i,E_i)$ be any $r$-decomposition of $K_n$.  Then there must be some $G_\ell$ that has ${|E_\ell|}\ge \frac {n(n-1)}{2r}$.  So by \cite[Theorem 1]{Kost84} with $k=\frac {n-1}{2r}$, \vspace{-2mm}
\[\h(G_1)+\dots+\h(G_r)\ge \h(G_\ell)\ge\frac {k}{270\sqrt{\log k}}\ge \frac {n-1}{540r\sqrt{\log{(n-1)}}}.\vspace{-2mm}\]
Thus for $n\ge 19$,  $\ngsl{\h}(r;n)\ge  \frac 1{570r}\frac n{\sqrt {\log n}}$.

Almost all graphs $G$  of order $n$ 
have $\h(G)\le \frac n {\sqrt{\log n}}$ \cite[p. 308]{Kost84}.  Thus there exists an $r$-decomposition such that  \vspace{-2mm}
\[\h(G_1)+\dots+\h(G_r)\le  r\frac n {\sqrt{\log n}},\vspace{-2mm}\]
and $\ngsl{\h}(r;n)\le r \frac n{\sqrt {\log n}}$.
\epf

 The next two results establish upper bounds for the sum lower bound $\ngsl{\beta}(r;n)$ for $r\ge 3$ and  $\beta\in\{\mu, \nu, \xi, \tw, \la, \pw,\ppw\}$, beginning with path-width.  

\begin{thm}\label{thmpwsumlower}    For $r\ge 3$, 
$ \ngsl{\pw}(r;n)\le 3\lc\frac n 4 \rc$.
\end{thm}
\bpf  Consider  the case $n=4p$ and 
$r=3$ first, and define the following 3-decomposition (illustrated in Figure \ref{twcounterex}):  Partition the vertices of $K_n$ into four sets $V_i,i=1,\dots,4$ of $p$ vertices each.   The edges of $G_1$ are the edges within $V_1$, the edges within $V_4$, the edges between $V_1$ and $V_2$, and the edges between $V_3$ and $V_4$.  The edges of $G_2$ are the edges within $V_2$, the edges within $V_3$, the edges between $V_1$ and $V_3$, and the edges between $V_2$ and $V_4$.  The edges of $G_3$ are the edges  between $V_1$ and $V_4$, and the edges between $V_2$ and $V_3$.
\begin{figure}[!h]
\begin{center}
$G_1$\qquad\qquad\qquad$G_2$\qquad\qquad\qquad$G_3$\vspace{-1mm}
\caption{Schematic diagram of a decomposition of $K_{4p}$ with $\pw(G_1)+\pw(G_2)+\pw(G_3)=3p$, where the shaded regions have all edges and the unshaded regions have no edges
.}\label{twcounterex}
\end{center}\vspace{-20pt}
\end{figure}

Observe that  $G_1$ and $G_2$ each consist of two copies of the same graph $H$, so $\pw(G_1)=\pw(G_2)=\pw(H)$.  Since $H$ can be constructed by starting with a $K_p$ and adding $p$ additional vertices each adjacent to all the vertices of the original $K_p$, $\pw(H)=p$.  Note that $G_3$ consists of two copies of $K_{p,p}$ and $\pw(K_{p,p})=p$.  Therefore, \vspace{-2mm}
\[\frac 3 4 n = 3p=\pw(G_1)+\pw(G_1)+\pw(G_1)\ge \ngsl{\pw}(3;n).\vspace{-2mm}\]

For the case in which $n=4p+\ell$ with $1\le\ell\le 3$, assign $p+1$ vertices to $V_i, i=1,\dots,\ell$ and $p$ to the other sets.  Then $\pw(G_i)\le p+1=\lc\frac n 4 \rc$, and $3\lc\frac n 4\rc\ge \pw(G_1)+\pw(G_2)+\pw(G_3)\ge \ngsl{\pw}(3;n)$.
For the case of $r>3$, a degenerate decomposition can be used.
  \epf\vspace{-1mm}

\begin{cor}\label{thmallsumlower}  $\null$\vspace{-10pt}
\ben[(i)]
\item\label{thmallsumlower1} Let $\beta\in\{\tw, \la, \pw\}$.  For $r\ge 3$, \vspace{-5pt}
\[ \ngsl{\beta}(r;n)\le 3\lc\frac n 4\rc \qquad\mbox{ and }\qquad\ngsln{\beta}(r;n)\le 3\lc\frac n 4\rc+r-3.\vspace{-10pt}\]
\item\label{thmallsumlower3}   For $r\ge 3$, \vspace{-5pt} \[ \ngsl{\ppw}(r;n)\le 3\lc\frac n 4\rc +r \qquad\mbox{ and }\qquad\ngsln{\ppw}(r;n)\le 3\lc\frac n 4\rc+2r-3.\vspace{-10pt}\]
\een\end{cor}
\bpf  The first statement in \eqref{thmallsumlower1} follows from Theorem \ref{thmpwsumlower} together with the inequality $ \tw(G)\le \la(G)\le \pw(G)$ for all graphs $G$. The second statement in \eqref{thmallsumlower1} follows from the first statement by first constructing a degenerate $r$-decomposition based on the non-degenerate 3-decomposition in Theorem \ref{thmpwsumlower}, and then removing $r-3$ edges from $G_3$ and placing one of these edges in each $G_i, i=4,\dots,n$.
  Statement \eqref{thmallsumlower3} follows from  statement  \eqref{thmallsumlower1} and the fact that  $\ppw(G)\le\pw(G)+1$  for any graph $G$.
    \epf
 \vspace{-2mm}
 
 \begin{thm}
\label{thm:twblbd}
For $r\geq 2$ and $n\ge 1$, \vspace{-2mm}
\[\ngsl{\tw}(r;n)\geq rn-\frac{1}{2}r-\sqrt{(r^2-r)n^2-(r^2-r)n+\frac{1}{4}r^2}.\vspace{-2mm}\]
\end{thm}
\begin{proof}  
First note that every $k$-tree on $n$ vertices (necessarily $n\ge k+1$) has 
\[E_k:=\frac{k(k-1)}{2}+(n-k)k=-\frac{1}{2}k^2+(n-\frac{1}{2})k\]
edges. So if $|V(G)|=n$ and $\tw(G)=k$, then $|E(G)|\le E_k$. Suppose $K_n$ is decomposed as $\{G_i\}_{i=1}^r$, with $a_k$ graphs having $\tw(G_i)=k$ for $k=0,\dots,n-1$.  Since every edge of $K_n$ must be covered,   
\begin{equation}
\label{treecover}
\sum_{k=0}^{n-1}a_kE_k\geq \frac{n(n-1)}{2}. 
\end{equation}
So we are trying to minimize $\sum_{i=1}^r{\tw(G_i)}$ subject to inequality (\ref{treecover}).

We begin by defining $S_1$ and $S_2$ by\vspace{-1mm}
\bea
S_1&\!\!\!\!:=&\sum_{k=0}^{n-1}a_kk =\sum_{i=1}^r{\tw(G_i)};\\
S_2&\!\!\!\!:=&\sum_{k=0}^{n-1}a_kk^2.
\eea
Since $S_2$ can be viewed as the sum of $r=\sum_{k=0}^{n-1}a_k$ squares, by Cauchy-Schwarz inequality,
\[\left(\sum_{k=0}^{n-1}a_kk^2\right)\left(r\cdot 1^2\right)\geq \left(\sum_{k=0}^{n-1}a_kk\right)^2.\]
This means $S_2\geq \frac{S_1^2}{r}$. \vspace{-3mm}

\bea\sum_{k=0}^{n-1}a_kE_k&=&-\frac{1}{2}S_2+(n-\frac{1}{2})S_1\\[-1mm]
&\leq& -\frac{1}{2r}S_1^2+(n-\frac{1}{2})S_1.\vspace{-2mm}\eea
So inequality \eqref{treecover} implies 
\[-\frac{1}{2r}S_1^2+(n-\frac{1}{2})S_1-\frac{n(n-1)}{2}\geq 0.\]
By solving the inequality,
\[\sum_{i=1}^r{\tw(G_i)}=S_1\geq rn-\frac{1}{2}r-\sqrt{(r^2-r)n^2-(r^2-r)n+\frac{1}{4}r^2}.\]
So $\ngsl{\tw}(r;n)$ is also greater than or equal to this value.
\end{proof}

\begin{cor}
\label{cor:twblower}  Let $\beta\in\{\tw,\la,\pw,\ppw\}$.  
For $r\geq 2$, 
\[\liminf_{n\to\infty}\frac{\ngsl{\beta}(r;n)}n\geq r-\sqrt{r^2-r} > \frac{1}{2}.\]
\end{cor}
\bpf  For a fixed $r$, \[\lim_{n\to\infty}\frac{rn-\frac{1}{2}r-\sqrt{(r^2-r)n^2-(r^2-r)n+\frac{1}{4}r^2}}n=r-\sqrt{r^2-r}\ge \frac 1 2,\]
with the last inequality verified by simple algebra.
\epf\vspace{-2mm}

  \begin{cor}\label{thmCdVsumlower}  $\null$\vspace{-10pt}
\ben[(i)]
\item\label{thmallsumlower2}   For $r\ge 3$ and $n\ge 19$,  \vspace{-5pt}
\[\frac 1 {570 r}\frac n{\sqrt {\log n}}-r\le \ngsl{\nu}(r;n)\le \ngsln{\nu}(r;n)\le 3\lc\frac n 4\rc+r-3.\]
\item\label{thmallsumlower4} Let $\beta\in\{\mu,\xi\}$.  For $r\ge 3$ and $n\ge 19$,  \vspace{-5pt}
\[ \frac 1 {570 r}\frac n{\sqrt {\log n}}-r\le\ngsl{\beta}(r;n)\le \ngsln{\beta}(r;n)\le 3\lc\frac n 4\rc+2r-3.\vspace{-10pt}\]
\een\end{cor}
\bpf   The first inequality in each statement follows from Theorem \ref{etabarthm} and the inequalities $\eta(G)-1\le\nu(G)$ and  $\eta(G)-1\le\mu(G)\le\xi(G)$ for all graphs $G$.  The third inequality in  statement \eqref{thmallsumlower2} (respectively, \eqref{thmallsumlower4}) follows from the non-degenerate case in Corollary \ref{thmallsumlower}\eqref{thmallsumlower1} (respectively,  \ref{thmallsumlower}\eqref{thmallsumlower3}) and the fact that for a graph $G$ that has an edge, $\nu(G)\le  \la(G)$     \cite{CdVnu} (respectively, $\mu(G)\le\xi(G)\le\ppw(G)$   \cite{param}). 
    \epf\vspace{-2mm}

  Corollaries \ref{thmallsumlower} and \ref{thmCdVsumlower} provide our best upper bounds on $\ngsl{\beta}(r;n)$ as $n\to \infty$ for $r\ge 3$ and $\beta\in\{\mu,\nu, \xi, \tw, \la, \pw, \ppw\}$.  However, for $n$ in the range $2r\le n < 4r$ (and in some cases somewhat larger), the bound in the next theorem is better, and it is also used in Section \ref{sNGprodl}.
  
  \begin{thm}\label{oldthmallsumlower}  Let $\beta\in\{\mu,\nu, \xi, \tw, \la, \pw, \ppw\}$.  For $r\ge 2$ and $n\ge 2r$, \[ \ngsl{\beta}(r;n)\le \ngsln{\beta}(r;n)\le n-r,\] and $n-r$ is realized by a non-degenerate decomposition in which all but at most one of the graphs  are paths.
\end{thm}
\bpf Consider proper path-width first.  Define $K_{2r}$ to be the subgraph of $K_n$ induced by the vertices $[2r]$. 
We can partition the $r(2r-1)$ edges of $K_{2r}$ into $r$ paths each having $2r-1$ edges \cite{T}. Without loss of generality, one of the paths, which we call the ``last path," is $(1,2,\dots,2r)$.  Define the following non-degenerate $r$-decomposition of $K_{n}=([n],E)$: The $r-1$ paths of $K_{2r}$ that are not the last path, each taken together with $n-2r$ isolated vertices, denoted by $P_i=([n], E_i), i=1,\dots,r-1$, and $P_r=\left([n],E\setminus\left(\dot{\cup}_{i=1}^{r-1} E_i\right)\right)$; observe that $P_r$ includes the last path.  For every  $i=1,\dots,r-1$, $\ppw(P_i)=1$.  We construct $P_r$ as a  linear $(n-2r+1)$-tree  as follows:  
Begin with the $(n-2r+2)$-clique induced by   $\{2r-1,\dots,n\}$.  For $i=1,\dots,2r-2$, add vertex $2r-1-i$ adjacent to $2r-i$ and to each vertex in $\{2r+1,\dots,n\}$.  Thus  $\ppw(P_r)= n-2r+1$.
Since each $P_i$ has an edge, $\ngsln{\ppw}(r,n)\le (r-1)\cdot 1+(n-2r+1)=n-r$. 

Finally, $\tw(G)\le\la(G)\le\pw(G)\le\ppw(G)$, $\nu(G)\le\la(G)$, and $\mu(G)\le\xi(G)\le \ppw(G)$ for every graph $G$ that has an edge,  and the decomposition just constructed is non-degenerate. Thus we  have $
\ngsln{\beta}(r;n)\le n-r$ for  $\beta\in\{\mu,\nu, \xi, \tw, \la, \pw\}$. \epf

\section{Nordhaus-Gaddum Product Upper Bounds}
\label{sNGprodu}

To see the relation between sums and products, we use 
the inequality of arithmetic and geometric means (AM-GM inequality)\vspace{-2mm}
\[\prod_{i=1}^r a_i\leq r^{-r} \left(\sum_{i=1}^ra_i\right)^r\vspace{-2mm}\]
for nonnegative real numbers $a_1,a_2,\ldots ,a_r$.

To establish a lower bound for a product, we use a technical lemma. 

\begin{lem}
\label{lem:mainieq}
For fixed  integers $r\ge 2$ and $n\ge 1$, let $a_1,a_2,\ldots,a_r$ be integers such that $1\leq a_i\leq n$ for all $i$.  Let $\sigma=\sum_{i=1}^r a_i$ and $\pi=\prod_{i=1}^r a_i$.  Then $n^q\rho \leq \pi$, where 
$q=\left\lfloor\frac{\sigma-r}{n-1}\right\rfloor$ and $\rho=\sigma-r-q(n-1)+1$.
\end{lem}
\begin{proof}
Suppose $1<a_1\leq a_2<n$.  Observe that \vspace{-2mm}
\[(a_1-1)+(a_2+1)+\sum_{i=3}^ra_i=\sigma\vspace{-2mm}\]
but\vspace{-2mm}
\[(a_1-1)(a_2+1)\prod_{i=1}^ra_i<\pi.\vspace{-2mm}\]
Therefore, subject to a fixed $\sigma$, the value of $\pi$ is minimized when all values of $a_i$'s but at most one are either $1$ or $n$.  This is equivalent to solving for nonnegative integers $q$ and $\rho$ in\vspace{-2mm} \[\sigma=(r-1-q)\cdot 1+qn+\rho,\vspace{-2mm}\]
where $1\leq \rho < n$.  This equation can be rewritten as $\sigma-r=q(n-1)+(\rho-1)$ with $0\leq \rho-1< n-1$.  Therefore, $q=\left\lfloor\frac{\sigma-r}{n-1}\right\rfloor$ is uniquely determined by the division algorithm, and so is $\rho=\sigma-r-q(n-1)+1$.
\end{proof}

For all parameters $\beta$ discussed in this paper, $\beta(G)\leq n$ (where $n$ is the order of $G$).  Therefore, $\ngpu{\beta}(r;n)\leq n^r$.  For tree-width and its variants largeur d'arborescence, path-width and proper path-width, this upper bound is essentially the best we can do. 

\begin{thm}
\label{thm:twpu}
Let $\beta\in \{\tw,\la,\pw,\ppw\}$.  For all $r\ge 2$, $\ngpu{\beta}(r;n)=\ngpun{\beta}(r;n)=n^r-o(n^r)$.
\end{thm}
\begin{proof}
Since $\ppw(G)\ge\pw(G)\ge\la(G)\ge \tw(G)$ for all $G$, it is enough to prove $\ngpu{\tw}(r;n)= n^r-o(n^r)$.

By Observation \ref{rem:prodnondegen}, $\ngpu{\tw}(r;n)=\ngpun{\tw}(r;n)$. Theorem \ref{thmtwupper} shows that $\ngsun{\tw}(r;n)=rn-o(n)$, since when $n$ is large enough such that $\tw(r;n) \geq r n - 0.5 n$, the value of $\tw(r;n)$ can be achieved only by a non-degenerate decomposition (because $\tw(G) \leq n-1)$.
Fix $\epsilon$ with $0<\epsilon<1$.  When $n$ is large enough, there is a non-degenerate $r$-decomposition $G_1,\ldots ,G_r$ such that 
\[\tw(G_1)+\cdots +\tw(G_r)\geq (r-\epsilon)n.\]
Applying Lemma \ref{lem:mainieq} with $\sigma=(r-\epsilon)n$, we get $q=\lf\frac{(r-\epsilon)n-r}{n-1}\rf=\lf r-\frac{\epsilon n}{n-1}\rf=r-1$ for $n$ large enough, and $\rho=(r-\epsilon)n-(r-1)n=(1-\epsilon)n$, so
\[\tw(G_1)\cdots\tw(G_r)\geq (1-\epsilon)n^r.\]
Since $\epsilon>0$ can be taken arbitrarily small, $\ngpun{\tw}(r;n)= n^r-o(n^r)$. 
\end{proof}


\begin{thm}\label{thm:CdVpu} Let $\beta\in \{\xi,\nu,\mu,\h\}$.  For  $r\ge 2$, and any $n\ge 2\sqrt r$,  
\[t^{-r}n^r-o(n^r)\leq \ngpu{\beta}(r;n)\le r^{-\frac{r}{2}}n^r+o(n^r),\]
where $t:=\lc\trt(r)\rc=\lc \sqrt{2  r+\frac{1}{4}}-\frac{1}{2}\rc$.
\end{thm}
\begin{proof}
The upper bound comes immediately from Theorem \ref{thmrxiupper} and the AM-GM inequality.  To see the lower bound, we build a decomposition as follows.  Let $t$ be the minimum integer such that $r\leq \frac{t(t+1)}{2}$, i.e., $t=\lc \trt(r)\rc$.  Partition the vertex set into $t$ parts $V_1,V_2,\ldots ,V_t$ with $|V_i|\geq \lf \frac{n}{t}\rf$ for all $i$.  Consider $H_i$ as the subgraph with all edges in $V_i$, and $H_{i,j}$ as the subgraph with all edges between $V_i$ and $V_j$.  We have $\nu(H_i)\geq \eta(H_i)-1\geq \lf \frac{n}{t}\rf-1$ and $\nu(H_{i,j})\geq\eta(H_{i,j})-1\geq \lf \frac{n}{t}\rf$ (and similarly for $\mu$).  Take $r$ subgraphs out of those $H_i$ and $H_{i,j}$, then merge all remaining edges to one of the subgraphs.  This builds a $r$-decomposition with product at least 
\[\left(\lf \frac{n}{t}\rf-1\right)^r\!.\]
Therefore, $\ngpu{\beta}(r;n)\geq t^{-r}n^r-o(n^r)$. 
\vspace{-2mm}\end{proof}



\section{Nordhaus-Gaddum Product Lower Bounds}
\label{sNGprodl}

For $\beta\in \{\h,\mu,\nu,\xi,\tw,\la,\pw,\ppw\}$ we show that  the growth rate of  $\ngpln{\beta}(r;n)$ is $\Theta(n)$.

Recall that for $\beta\in \{\tw,\la,\pw,\ppw\}$, $\beta(\overline{K_n})=0$, so $\ngpl{\beta}(r;n)=0$; therefore, for these parameters we focus on the non-degenerate case.  
We first prove a technical result that allows us to convert a sum lower bound to a product lower bound for non-degenerate decompositions.

\begin{thm}
\label{thm:sum2prod}
Suppose that  for $r\ge 2 $ and every graph $G$ on $n$ vertices that has an edge,  $1\le \beta(G)\le n$ and $\ngsln{\beta}(r;n)< n+r-1$. Then \vspace{-1mm}
\[ \ngsln{\beta}(r;n)-r+1\le \ngpln{\beta}(r;n)\vspace{-1mm}\]
This formula also applies to $\ngpl{\beta}(r;n)$ if the hypotheses are satisfied for all graphs $G$ (without the restriction of having an edge).  \end{thm}
\begin{proof}  Let $G_1,\ldots ,G_r$ be a non-degenerate $r$-decomposition that achieves $\ngpln{\beta}(r;n)$.  Pick $1\leq a_i\leq \beta(G_i)$ such that $\sum_{i=1}^ra_i=\ngsln{\beta}(r;n)<n+r-1$.  Now apply Lemma \ref{lem:mainieq} with $\sigma=\ngsln{\beta}(r;n)$.  Then $q=\left\lfloor\frac{\sigma-r}{n-1}\right\rfloor=0$ and $\rho=\sigma-r+1$.  Therefore,\vspace{-2mm}
\[n^q\rho=\sigma-r+1\leq \prod_{i=1}^ra_i\leq \ngpln{\beta}(r;n).\vspace{-2mm}\]
When $\beta(G)\geq 1$ for all $G$, the same argument works for $\ngpl{\beta}(r;n)$.
\end{proof}

\begin{thm}\label{thm:twprod}
Let $\beta\in \{\tw,\la,\pw,\ppw\}$.  \vspace{-2mm}
\ben[(i)] 
\item\label{421i} For $n\ge 4$,\vspace{-2mm} \[\ngpln{\beta}(2;n)= n-3.\vspace{-2mm}\]
\item For a fixed $r\ge 3$ and $n$ large enough, \vspace{-2mm}
\[\frac{n}{2}-r+1\leq \ngpln{\beta}(r;n)\leq n-2r+1.\vspace{-2mm}\]
\een
\end{thm}
\begin{proof}
By the non-degenerate $r$-decomposition of $K_n$ into $r-1$ paths and one large piece in Theorem \ref{oldthmallsumlower}, 
\[\ngpln{\ppw}(r;n)\leq n-2r+1.\]
Note that $\ngsln{\tw}(r;n)\geq \ngsl{\tw}(r;n)$ by definition.  By \cite{EGR11, JW12}, $\ngsl{\tw}(2;n)\ge n-2$, and for $r\ge 3$ and  $n$ large enough, $\ngsl{\tw}(r;n)\geq \frac{n}{2}$ by Corollary \ref{cor:twblower}.  Consequently, Theorem \ref{thm:sum2prod} implies 
$\ngpln{\tw}(2;n)\geq n-3$, and $\ngpln{\tw}(r;n)\geq \frac{n}{2}-r+1$ for $r\ge 3$ and $n$ large enough.
\end{proof}

Note that  Theorem \ref{thm:twprod}\eqref{421i} was established in \cite{H15IMA} but non-degeneracy was implicitly assumed and should have been stated. Next we consider the Hadwiger number.  Since $\eta(G)=1$ if and only if $G$ has no edges, both the general and  non-degenerate decompositions are of interest, and  the product lower bounds  have  different values in the case $r=2$.

\begin{rem}\label{rem:degenHadprodlower}  
It is known \cite{K89} 
that $\ngpln{\h}(2;n)\geq \left\lceil \frac{3n-5}{2}\right\rceil$.  If one of the two parts has no edge, then the decomposition becomes $K_n$ and $\overline{K_n}$, and the product is $n$.  When $n\geq 4$, $n\le \left\lceil \frac{3n-5}{2}\right\rceil$, so $\ngpl{\h}(2;n)=n$ 
(because the decomposition  $K_n,\overline{K_n}$ also provides an upper bound).   So  $\ngpln{\h}(2;n)> \ngpl{\h}(2;n)$ for $n\ge 5$.  By checking small cases, we see that $\ngpl{\h}(2;n)=n$ for all $n$.
\end{rem}

For a graph $G$ on $n$ vertices, Balogh and Kostochka, building on work of Duchet and  Meyniel \cite{DM82} and Fox \cite{JFox}, showed in \cite{BK11}  that\vspace{-2mm} \[0.513 n\le \h(G)\omega(\overline{G}) .\vspace{-2mm}\] 

\begin{thm}\label{thm:etaprod} 
For all $r\geq 2$ and  $n\ge 1$, \vspace{-2mm}
\[(0.513)^{r-2}n \leq \ngpl{\h}(r;n)\leq n.\vspace{-2mm}\]
\end{thm}
\begin{proof}
The upper bound is  achieved by one complete graph and $r-1$  empty graphs. 

  We prove the lower bound by induction.  For base case $r=2$, we already know $\ngpl{\h}(2;n)=n$.  Assuming $\ngpl{\h}(r-1;n)\geq (0.513)^{r-3}n$, we consider a $r$-decomposition $G_1,\ldots ,G_r$.  Since $\h(G_1)\omega(\overline{G_1})\geq 0.513 n$, there is a clique in $\overline{G_1}$ on the vertex set $W$ with $|W|\geq 0.513\frac{n}{\h(G_1)}$.  Now
\bea
\h(G_1)\cdots \h(G_r) & \geq& \h(G_1)\h(G_2[W])\cdots \h(G_r[W])\\
&\geq &\h(G_1)\cdot (0.513)^{r-3}|W|\geq (0.513)^{r-2}n,
\eea
since $G_2[W],\ldots,G_r[W]$ form an $(r-1)$-decomposition of the clique on $W$.  
\end{proof}


\begin{rem}\label{rem:etapl}
 It is known  that $\ngpl{\chi}(r;n)=n$ for all  $r$ and $n$ \cite[p. 277]{multiNG}. 
On the other hand, Hadwiger's conjecture states that $\h(G)\geq \chi(G)$ for all graph $G$.  Therefore, Hadwiger's conjecture  implies $\ngpl{\h}(r;n)\ge n$, and thus implies Conjecture \ref{etapl}, so any counterexample to Conjecture \ref{etapl} would disprove Hadwiger's conjecture.
\end{rem}

\begin{cor}\label{thm:etaprodnd} 
For all $r\geq 2$ and  $n\ge 2r$,\vspace{-2mm}
\[(0.513)^{r-2}n \leq \ngpln{\h}(r;n)\leq 2^{r-1}(n-2r+2).\]
 \end{cor}

\begin{proof}
The lower bound is by Theorem \ref{thm:etaprod} and the fact $\ngpln{\h}(r;n)\geq \ngpl{\h}(r;n)$.
On the other hand, let $P_1,P_2,\ldots ,P_r$ be the $r$-decomposition  in Theorem \ref{thmallsumlower}.  Recall that the Hadwiger number of a path is $2$ whenever it has more than one vertex.  Since $P_i$ is a path for $i=1,\ldots ,r-1$ and $\h(P_r)\leq \ppw(P_r)+1=n-2r+2$, we have $\ngpln{\h}(r;n)\leq 2^{r-1}(n-2r+2)$.
\end{proof}


A peculiarity of the parameters $\xi, \nu$, and $\mu$ is that $\beta(P_n)=\beta(\overline{K_n})=1$ for $\beta\in\{\xi, \nu,\mu\}$ (and $n\ge 2$ in the case of $\mu$). Thus $\ngpl{\beta}$ is not optimized on $\overline{K_n}$ for these parameters, and we  focus on the non-degenerate versions.

\begin{cor}\label{cor:etaprodnd2munuxi} Let   $\beta\in\{\nu,\xi,\mu\}$.  For $r\geq 2$ and $n\ge 2r$,\vspace{-2mm}
 \[\frac n {2^{2r-2}}\leq\ngpln{\beta}(r;n)\leq n-2r+1.\vspace{-2mm}\]
\end{cor}
\bpf For $\beta\in\{\nu,\xi,\mu\}$, $ \h(G)-1\le  \beta(G)\le \ppw(G)$ when $G$ has an edge.  
The result then follows from Theorem \ref{thm:twprod} and Corollary \ref{thm:etaprodnd}, since $\beta(G)\geq \h(G)-1\geq \frac{1}{2}\h(G)$ and $ \ngpln{\h}(r;n)\ge (0.513)^{r-2}n \geq \frac n{2^{r-2}}$.
\vspace{-5mm}\epf

\subsection*{Acknowledgments}
This research began while the authors were general members in residence at the Institute for Mathematics and its Applications, and they thank the IMA both for financial support and for a wonderful research environment.

\end{document}